\documentclass[12pt, reqno]{amsart}
\usepackage{amsmath, amsthm, amscd, amsfonts, amssymb, graphicx, color}
\usepackage[bookmarksnumbered, colorlinks, plainpages]{hyperref}
\hypersetup{colorlinks=true,linkcolor=red, anchorcolor=green, citecolor=cyan, urlcolor=red, filecolor=magenta, pdftoolbar=true}

\usepackage{amsfonts}
\usepackage{color,graphicx,shortvrb}
\usepackage{enumerate}
\usepackage{amsmath}
\usepackage{amssymb}

\newtheorem{thm}{Theorem}
\newtheorem{corollary}[thm]{Corollary}
\newtheorem{lemma}[thm]{Lemma}

\theoremstyle{definition}

\theoremstyle{remark}
\newtheorem{remark}[thm]{Remark}
\newtheorem{problem}[thm]{Problem}

\numberwithin{equation}{section}

\usepackage[latin 1]{inputenc}

\begin{document}
\title[Linear maps anti-derivable at zero]{Linear maps which are anti-derivable at zero}

\author[D.A. Abulhamil]{Doha Adel Abulhamil}
\address{(Current Address) Mathematics Departments, College of Sciences, King Abdulaziz University, P.O. Box 9039 Jeddah 21413, Saudi Arabia.
}
\email{dabulhamil0001@stu.kau.edu.sa
}
\author[F.B. Jamjoom]{Fatmah B. Jamjoom}
\address{Mathematics Departments, College of Sciences, King Abdulaziz University, P.O. Box 9039 Jeddah 21413, Saudi Arabia}
\email{fjamjoom@kau.edu.sa}
\author[A.M. Peralta]{Antonio M. Peralta}
\address{Departamento de An{\'a}lisis Matem{\'a}tico, Universidad de Granada, 
Facultad de Ciencias 18071, Granada, Spain}
\email{aperalta@ugr.es}


\begin{abstract} Let $T:A\to X$ be a bounded linear operator, where $A$ is a C$^*$-algebra, and $X$ denotes an essential Banach $A$-bimodule. We prove that the following statements are equivalent:
\begin{enumerate}[$(a)$]\item $T$ is anti-derivable at zero (i.e. $ab =0$ in $A$ implies $T(b) a + b T(a)=0$);
\item There exist an anti-derivation $d:A\to X^{**}$ and an element $\xi \in X^{**}$ satisfying $\xi a = a \xi,$ $\xi [a,b]=0,$
$T(a b) = b T(a) + T(b) a - b \xi a,$ and $T(a) = d(a) + \xi a,$ for all $a,b\in A$.
\end{enumerate}
We also prove a similar equivalence when $X$ is replaced with $A^{**}$. This provides a complete characterization of those bounded linear maps from $A$ into $X$ or into $A^{**}$ which are anti-derivable at zero. We also present a complete characterization of those continuous linear operators which are $^*$-anti-derivable at zero.
\end{abstract}

\maketitle
\section{Introduction}

Let us begin this note by formulating a typical problem in recent studies about preservers. Suppose $X$ is a Banach $A$-bimodule over a complex Banach algebra $A$. A derivation from $A$ to $X$ is a linear mapping $D: A\to X$ satisfying the following algebraic identity \begin{equation}\label{eq fundamental identity of derivations} D(a b) = D(a) b + a D(b), \ \ \forall(a,b)\in A^2.
\end{equation} A derivation $D$ is called \emph{inner} if there exists $x_0\in X$ such that $D(a) = \delta_{x_0} (a) = [a, x_0] = a x_0 - x_0 a $ for all $a\in A$.\smallskip

A typical challenge on preservers can be posed in the following terms:
\begin{problem}\label{problem general problem}  Suppose $T: A\to X$ is a linear map satisfying \eqref{eq fundamental identity of derivations} only on a proper subset $\mathfrak{D}\subset A^2$. Is $T$ a derivation?
\end{problem}

There is no need to comment that the role of the set $\mathfrak{D}$ is the real core of the question. A typical example is provided by the set $\mathfrak{D}_z :=\{ (a,b)\in A^2 : a b =z\},$ where $z$ is a fixed point in $A$. A linear map $T: A\to X$ is said to be a \emph{derivation at a point} $z\in A$ if the identity \eqref{eq fundamental identity of derivations} holds for every $(a,b)\in\mathfrak{D}_z$. In the literature a linear map which is a derivation at a point $z$ is also called \emph{derivable at $z$}.\smallskip

Let us point out that there exist linear maps which are derivable at zero but they are not derivations (for example, the identity mapping on a complex Banach algebra is a derivation at zero but it is not a derivation).\smallskip

If $T:A\to B$ is a linear mapping from $A$ into another Banach algebra satisfying $T(ab) = T(a) T(b)$ for all $(a,b)\in\mathfrak{D}_z$ we say that $T$ is a \emph{homomorphism at the point $z$}. Linear maps which are Jordan ($^*$-)derivations, or generalized ($^*$-)derivations, or triple derivations, or (Jordan $^*$)-homomorphisms at a point can be defined in similar terms. We understand that term ``$*$-'' is only employed when the involved structures are equipped with an involution.\smallskip

Let us simply observe that a linear map $T$ between Banach algebras is a homomorphism at zero if and only if it preserves zero products (i.e., $ab =0$ implies $T(a) T(b)=0$). We find in this way a natural link with the results on zero products preservers (see, for example, \cite{AlBreExVill09, ArJar03, BurFerGarMarPe2008, CheKeLeeWong03, LeuTsaiWong12, LeuWong10, LiuChouLiaoWong2018, LiuChouLiaoWong2018b, Tsai11, TsaiWong10, Wolff94, Wong2005, Wong2007} for additional details and results). M.J. Burgos, J.Cabello-S{\'a}nchez and the third author of this note explore in \cite{BurCabSanPe2016} those linear maps between C$^*$-algebras which are $^*$-homomorphisms at certain points of the domain, for example, at the unit element or at zero. We refer to \cite{EssPe2018, Houqi, JingLuLi2002, ZhangHouQi2014, ZhangHouQi2014b, Zhu2007, ZhuXio2002, ZhuXio2005, ZhuXio2007, ZhuXion2008} and \cite{ZhuZhao2013} for additional related results.\smallskip

According to the standard terminology (cf. \cite{AyuKudPe2014,AlBreExVill09,BurFerGarPe2014,BurFerPe2014,JamPeSidd2015,LiPan}), we shall say that a linear operator $G$ from a Banach algebra $A$ into a Banach $A$-bimodule $X$ is a \emph{generalized derivation} if there exists $\xi\in  X^{**}$ satisfying $$G(ab) = G(a)  b + a  G(b) - a \xi b \hbox{ ($a, b \in A$).}$$ Every derivation is a generalized derivation, however there exist generalized derivations which are not derivations. This notion is very useful when characterizing (generalized) derivations in terms of annihilation of certain products of orthogonal elements (see, for example, Theorem 2.11 in \cite[\S 2]{AyuKudPe2014}). The just quoted reference \cite{AyuKudPe2014} contains an illustrative survey on local, 2-local and generalized derivations.\smallskip

Let us revisit some recent achievements on maps derivable at certain points. For example, every continuous linear map $\delta$ on a von Neumann algebra is a generalized derivation whenever it is derivable at zero. If we additionally asume $\delta(1)=0$, we can conclude that $\delta$ is a derivation (see \cite[Theorem 4]{JingLuLi2002}). Furthermore, for an infinite dimensional Hilbert space $H$, a linear map $\delta : B(H) \to B(H)$ which is a generalized Jordan derivation at zero, or at 1, is a generalized derivation, even if $\delta$ is not assumed to be a priori continuous (cf. \cite{Jing2009}). J. Zhu, Ch. Xiong, and P. Li prove in \cite{ZhuXiLi} a significant result showing that, for any Hilbert space $H$, a linear map $\delta :B(H)\to B(H)$ is a derivation if and only if it is a derivation at a non-zero point in $B(H)$ (see \cite{Lu2009} for another related result). \smallskip

H. Ghahramani, Z. Pan \cite{GhaPan2018} and B. Fadaee and H. Ghahramani \cite{FaaGhahra2019} have recently considered certain variants of Problem \ref{problem general problem} in their studies of continuous linear operators from a C$^*$-algebra $A$ into a Banach $A$-bimodule $X$ behaving like derivations or anti-derivations at elements in a certain subset of $A^2$ determined by orthogonality conditions. Let us detail the problem.

\begin{problem}\label{problem anti-derivable mapps} Let $T:A\to X$ be a continuous linear operator which is anti-derivable at zero, i.e.,  \begin{equation}\label{eq antider} \hbox{ $T(ab) = T(b) a + b T(a)$  for all $(a,b)\in \mathfrak{D}_0$.}
\end{equation}  Is $T$ an anti-derivation or expressible in terms of an anti-derivation?
\end{problem}

Clearly, a mapping $\mathcal{D} : A\to X$ is called an \emph{anti-derivation} if the identity \eqref{eq antider} holds for every $(a,b)\in A^2$. If $A$ is a C$^*$-algebra, a $^*$-derivation (respectively, a $^*$-anti-derivation) from $A$ into itself, or into $A^{**}$, is a derivation (respectively, an anti-derivation) $d: A\to A$ satisfying $d(a^*) = d(a)^*$ for all $a\in A$.\smallskip

Concerning Problem \ref{problem general problem}, B. Fadaee and H. Ghahramani prove in \cite[Theorem 3.1]{FaaGhahra2019} that for a continuous linear map $T : A \to A^{**}$, where $A$ is a C$^*$-algebra, the following statements hold:
\begin{enumerate}[$(a)$]\item $T$ is derivable at zero if and only if there is a continuous derivation $d : A \to  A^{**}$ and an element $\eta\in  \mathcal{Z}(A^{**})$ (the center of $A^{**}$) such that $T(a) = d(a) + \eta a$ for all $a \in A$;
\item $T$ is r- $^*$-derivable at zero (that is, $ab^* = 0 \Rightarrow$ $a T(b)^* + T(a) b^* = 0$) if and only if there is a continuous $^*$-derivation $d : A \to A^{**}$ and an element $\eta \in A^{**}$ such that $T(a) = d(a) + \eta a$ for all $a \in A$ ($\eta$ need not be central).
\end{enumerate}

H. Ghahramani, Z. Pan also considered a variant of Problem \ref{problem general problem} in \cite{GhaPan2018} in the context of (complex Banach) algebras which are zero product determined. We recall that an algebra $A$ is called \emph{zero product determined} if for every linear space $Y$ and every bilinear map $V : A\times A\to Y$ satisfying $V (x,y) =0$ for every $x,y\in A$ with $x y =0$, there exists a linear map $T : A\to Y$ such that $V(x, y) = T(xy)$ for all $x,y\in A$. Bre\v{s}ar showed in \cite[Theorem 4.1]{Bres12} that every unital algebra $A$ (algebraically) generated by its idempotents, is zero product determined. Since this is the case of $B(H)$ for any infinite dimensional complex Hilbert space $H$ (see \cite[Theorem 1]{PearTopp67}), a property which is also enjoyed by properly infinite von Neumann algebras \cite[Theorem 4]{PearTopp67}, Bunce-Deddens algebras, irrotational rotation algebras, simple unital AF C$^*$-algebras with finitely many extremal states, UHF C$^*$-algebras, unital simple C$^*$-algebras of real rank zero with no tracial states \cite[Corollary 4.9]{Mar02}, \cite[Theorem 4.6]{MarMur98}, properly infinite C$^*$-algebras \cite[Corollary 2.2]{LinMat}, and von Neumann algebras of type II$_1$ \cite[Theorem 2.2$(a)$]{GoldPas92}, all these algebras are zero product determined.\smallskip

A \emph{Jordan algebra} is a non-necessarily associative algebra $B$ over a field whose multiplication, denoted by $\circ$, is commutative and satisfies the so-called \emph{Jordan identity} $$( a \circ b ) \circ a^2 = a\circ ( b\circ a^2 ) \  \ (a,b\in B).$$ Every associative algebra is a Jordan algebra when equipped with the natural Jordan product given by $a\circ b = \frac12 (a b + ba)$. A Jordan derivation from $B$ into a Jordan $B$-module $X$ is a linear mapping $D: B\to X$ satisfying $D(a\circ b) = D(a) \circ b + a \circ D(b)$ for all $a,b\in B$. For the basic background on Jordan algebras, Jordan modules and  Jordan derivations the reader is referred to \cite{HOS,HejNik96} and the references therein.\smallskip

In what concerns Problem \ref{problem general problem}, Ghahramani and Pan proved in \cite[Theorem 3.1]{GhaPan2018} that for any zero product determined unital $^*$-algebra $A$, and every linear mapping $T: A\to A$ the following statements hold:
\begin{enumerate}[$(i)$]\item $T$ is derivable at zero if and only if there is a derivation $d : A\to A$ and an element $\eta\in  \mathcal{Z}(A)$ such that $T(a) = d(a) + \eta a$ for all $a \in A$;
\item $T$ is r- $^*$-derivable at zero (that is, $ab^* = 0 \Rightarrow$ $a T(b)^* + T(a) b^* = 0$) if and only if there is a $^*$-derivation $d : A \to A$ and an element $\eta \in A$ such that $T(a) = d(a) + \eta a$ for all $a \in A$ ($\eta$ need not be central).
\end{enumerate}

When considering Problem \ref{problem anti-derivable mapps} and maps which are anti-derivable at zero, the available conclusions are less determinate. Concretely, assuming that $A$ is a C$^*$-algebra, Theorem 3.3 in \cite{FaaGhahra2019} proves that for any continuous linear map $T : A \to A^{**}$ the following statements hold:\begin{enumerate}[$(i)$]
\item If $T$ is anti-derivable at zero there is a continuous derivation $d : A \to  A^{**}$ and an element $\eta \in \mathcal{Z}(A^{**})$ such that $T(a) = d(a) + \eta a$ for all $a \in A$;
\item If $T$ is r-$^*$-anti-derivable at zero (i.e., $ab^* = 0 \hbox{ in } A \ \Rightarrow T(b)^* a + b^* T(a) = 0$) there is a continuous $^*$-derivation $d : A \to A^{**}$ and an element $\eta \in A^{**}$ such that $T(a) = d(a) + a \eta$ for all $a\in A$  ($\eta$ need not be central).
\end{enumerate}

If $A$ is a zero product determined unital $^*$-algebra and $T: A\to A$ is a linear mapping, Theorem 3.4 in \cite{GhaPan2018} proves the following statements.
\begin{enumerate}[$(i)$]
\item If $T$ is anti-derivable at zero there is a Jordan derivation $d : A \to  A$ and an element $\eta \in \mathcal{Z}(A)$ such that $T(a) = d(a) + \eta a$ for all $a \in A$;
\item If $T$ is r-$^*$-anti-derivable at zero there is a Jordan $^*$-derivation $d : A \to A$ and an element $\eta \in A$ such that $T(a) = d(a) + a \eta $ for all $a\in A$  ($\eta$ need not be central).
\end{enumerate}

In view of the previous result it is natural to ask whether there exists a full characterization of those (continuous) linear maps which are ($^*$-)anti-derivable at zero in pure algebraic terms. The main aim of this note is to complete our knowledge on these clases of continuous linear maps and to fill a natural gap which has not been fully covered. Our first main conclusion is contained in Theorem \ref{t characterization of anti-derivations} where it is established that for each bounded linear operator $T$ from a C$^*$-algebra $A$ into an essential Banach $A$-bimodule $X$ the following statements are equivalent:\begin{enumerate}[$(a)$]\item $T$ is anti-derivable at zero;
\item There exist an anti-derivation $d:A\to X^{**}$ and an element $\xi \in X^{**}$ satisfying $\xi a = a \xi,$ $\xi [a,b]=0,$ $T(a b) = b T(a) + T(b) a - b \xi a,$ and $T(a) = d(a) + \xi a,$ for all $a,b\in A$.
\end{enumerate}
It is further shown that if $A$ is unital, or if $X$ is a dual Banach $A$-bimodule, statement $(b)$ above can be replaced with \begin{enumerate}[$(b^\prime)$]
\item There exist an anti-derivation $d:A\to X$ and an element $\xi \in X$ satisfying $\xi a = a \xi,$ $\xi [a,b]=0,$ $T(a b) = b T(a) + T(b) a - b \xi a,$ and $T(a) = d(a) + \xi a,$ for all $a,b\in A$.
\end{enumerate}

A similar conclusion holds when $X$ is replaced with $A^{**}$.\smallskip

In section \ref{sec: continuous linear maps which are *-anti-derivable at zero} we consider a C$^*$-algebra $A$ and an essential Banach $A$-bimodule equipped with an $A$-bimodule involution $*$ (i.e., a continuous conjugate linear mapping $x \mapsto x^*$ satisfying $(x^*)^*=x,$ $(a x)^* = x^* a^*$ and $(x a)^* = a^* x^*$, for all $a\in A$, $x\in X$). We give several natural examples of bimodule involutions. A Banach $A$-bimodule equipped with an $A$-bimodule involution will be called a Banach $^*$-$A$-bimodule. Suppose $T : A\to X$ is a linear mapping from a C$^*$-algebra into a Banach $^*$-$A$-bimodule. We shall say that $T$ is {r-$^*$-anti-derivable at zero} (respectively, {l-$^*$-anti-derivable at zero}) if $ab^* = 0$ in $A$ implies $T(b)^* a + b^* T(a) = 0$ in $X$ (respectively, $a^* b = 0 \hbox{ in } A \Rightarrow T(b) a^* + b T(a)^* = 0$ in $X$).\smallskip

Let $T: A\to X$ be a bounded linear operator where $A$ is a C$^*$-algebra and $X$ is an essential Banach $^*$-$A$-bimodule. In Theorem \ref{t maps r-*-anti-derivable at zero} we prove the equivalence of the following statements:\begin{enumerate}[$(a)$]\item $T$ is r-$^*$-anti-derivable at zero {\rm(}i.e., $ab^* = 0 \hbox{ in } A \Rightarrow T(b)^* a + b^* T(a) = 0$ in $X${\rm)};
\item There exists a $^*$-derivation $d:A\to X^{**}$ and an element $\xi \in X^{**}$ satisfying the following properties: \begin{enumerate}[$(i)$]\item  $d([a,b]) + [a,b] \xi + \xi^* [a,b]=0,$ for all $a,b\in A$;
\item $T(a b) = a T(b) + T(a) b - a \xi b,$ and $T(a) = d(a) + a \xi $ for all $a,b\in A$.
\end{enumerate}
\end{enumerate}

The conclusion in $(b)$ can be improved if $A$ is unital or if $X$ is a dual Banach $A$-bimodule. Finally, a complete characterization of those bounded linear operators $T: A\to X$ which are l-$^*$-anti-derivable at zero is presented in Theorem \ref{t maps l-*-anti-derivable at zero} (see also Corollary \ref{c maps l-*-anti-derivable at zero bidual}).\smallskip

In Section \ref{sec: are there anti-derivations} we take a closer look at anti-derivations from a general C$^*$-algebra $A$ into a Banach $A$-bimodule. Theorem \ref{t Cstar generated by commutators and all} shows that these maps are in general very scarce.

\section{Are there anti-derivations on C$^*$-algebras?}\label{sec: are there anti-derivations}

If one is interested on the study of linear maps from a C$^*$-algebra $A$ into a Banach $A$-bimodule which are anti-derivable at zero, a first natural step is to explore the class of anti-derivations on C$^*$-algebra. For this purpose, we initiate our study by paying some attention to anti-derivations. An anti-derivation from an (associative) algebra $A$ into a Banach $A$ bimodule $X$ is a linear mapping $d: A\to X$ satisfying $d(ab) = d(b) a + b d(a)$  for all $(a,b)\in A^2$. An example seems to be welcome. Let us fix an element $x_0\in X$ satisfying \begin{equation}\label{eq commutation with commutators} x_0 [a,b] =[a,b] x_0.
 \end{equation} for all $a,b\in A$, where $[a,b]= (ab -ba)$ denotes the Lie product or commutator of $a$ and $b$. The prototype of derivation from $A$ into $X$ is given by $\delta_{x_0} : A\to X$, $\delta_{x_0} (a) = [a,x_0]$ ($a\in A$). The assumption \eqref{eq commutation with commutators} implies that $a b x_0 - x_0 a b = b a x_0 - x_0 b a $ and thus $$ \delta_{x_0} (a b) = a b x_0 - x_0 a b = b a x_0 - x_0 b a = b x_0 a - x_0 b a + b a x_0 - b x_0 a  =  \delta_{x_0} (b) a + b \delta_{x_0} (a) ,$$ for all $a,b\in A$, witnessing that $\delta_{x_0}$ is a derivation and an anti-derivation. But, does such an element $x_0$ exist with the additional property that $\delta_{x_0}\neq 0$? Let us observe that $\delta_{x_0}$ also satisfies the following property
\begin{equation}\label{eq starring identity of anti-derivations from example} \delta_{x_0} ([a,b]) = [[a,b],x_0] =0, \hbox{ for all } a,b\in A.
\end{equation} We shall see next that the identity in \eqref{eq starring identity of anti-derivations from example} actually characterizes anti-derivations.

\begin{lemma}\label{l characterization anti-derivations} Let $\delta: A\to X$ be a linear mapping from an associative algebra into an $A$-bimodule. Then the following statements are equivalent:\begin{enumerate}[$(a)$]\item $\delta$ is a derivation and $\delta ([a,b])=0,$ for all $a,b\in A;$
\item $\delta$ is an anti-derivation and $\delta ([a,b])=0,$ for all $a,b\in A.$
\end{enumerate}
\end{lemma}

\begin{proof} The equivalence is clear by just observing that $\delta ([a,b])=0,$ for all $a,b\in A$ if and only if $\delta (ab) = \delta (ba)$, for all $a,b\in A$.
\end{proof}

A central result in the theory of derivations on C$^*$-algebras was established by J.R. Ringrose who proved that every (associative) derivation from a C$^*$-algebra $A$ to a Banach $A$-bimodule $X$ is (automatically) continuous (compare \cite{Ringrose72}).  B.E.
Johnson established in \cite{John96} another result to have in mind by proving that every bounded Jordan derivation from a C$^*$-algebra $A$ into a Banach $A$-bimodule $X$ is an associative derivation. By a result due to B. Russo and the third author of this note we know that every Jordan derivation from $A$ into $X$ is continuous \cite[Corollary 17]{PeRu2014}, consequently every Jordan derivation from $A$ into $X$ is a derivation.\smallskip

Let $\delta: A\to X$ be an anti-derivation from a C$^*$-algebra into a Banach $A$-bimodule. It is clear that $\delta(a\circ b) = \delta(a)\circ b + a\circ \delta(b)$ ($a,b\in A$), and hence $\delta$ is a Jordan derivation. It follows from the arguments in the previous paragraph that $\delta$ is a continuous derivation. So, every anti-derivation from a C$^*$-algebra $A$ into a Banach $A$-bimodule is continuous and a derivation, therefore a linear mapping $T:A\to X$ is an anti-derivation if and only if it is a derivation and $T([a,b])=0$ for all $a,b\in A$.\label{eq antiderivation is a derivation vanishing on commutators} So, the natural question is whether there exist non-trivial derivations vanishing on all commutators.\smallskip

Let $A$ be a C$^*$-algebra. According to the usual notation, we write $\mathfrak{c}(A) :=\{ [a, b] : a, b \in A\}$. The available literature contains a wide list of papers conducted to determine when an element in a C$^*$-algebra can be expressed as a finite sum of commutators (see, for example, \cite{FackdelaHarpe80, Fack82, Halpern69, Marc2006, Marc2010, MarMur98, Pop2002} and references therein). Let us remark two concrete results Th. Fack and P. de la Harpe showed in \cite{FackdelaHarpe80} that in any finite von Neumann algebra $M$ with central trace $\tau$, an element $x \in M$ with $\tau (x) = 0$ can be expressed as a sum of 10 commutators. In another remarkable result H. Halpern proved that every element $a$ in a properly infinite von Neumann algebra $M$ can be written as the sum of two commutators (cf. \cite[Theorem 3.10]{Halpern69}).\smallskip

Let $X$ be a Banach $A$-bimodule, where $A$ is a C$^*$-algebra. In this note we shall deal with the bidual, $X^{**}$, of $X$, and we shall regard it as a Banach $A^{**}$-bimodule. For this purpose we shall refresh our knowledge on Arens extensions and Arens regularity (cf. \cite{Arens51}). Let $m : X\times Y \to Z$ be a bounded bilinear map where $X$, $Y$ and $Z$ are Banach spaces. According to the construction defined by R. Arens, we define $m^* (z^\prime,x) (y) := z^\prime (m(x,y))$ $(x\in X, y\in Y, z^\prime\in Z^*)$. We obtain in this way a bounded bilinear mapping $m^*: Z^*\times X \to Y^*.$ The same method can be applied to define $m^{**} = (m^{*})^{*}$ and $m^{***}: X^{**}\times Y^{**} \to Z^{**}.$ The mapping $x^{\prime\prime}\mapsto  m^{***}(x^{\prime\prime} , y^{\prime\prime})$ is weak$^*$ to weak$^*$ continuous whenever we fix  $y^{\prime\prime} \in  Y^{**}$, and the mapping $y^{\prime\prime}\mapsto  m^{***}(x, y^{\prime\prime})$ is weak$^*$ to weak$^*$ continuous for every $x\in  X$. The previous construction can be applied to the transposed mapping $m^{t} : Y\times X\to Z,$ $m^{t} (y,x) = m(x,y),$ and we define an extension $m^{t***t}: X^{**}\times Y^{**} \to Z^{**}.$ Now, the mapping $x^{\prime\prime}\mapsto  m^{t***t}(x^{\prime\prime} , y)$ is weak$^*$ to weak$^*$ continuous whenever we fix  $y \in  Y$, and the mapping $y^{\prime\prime}\mapsto  m^{t***t}(x^{\prime\prime}, y^{\prime\prime})$ is weak$^*$ to weak$^*$ continuous for every $x^{\prime\prime}\in  X^{**}$. It should be remarked that the mappings $m^{t***t}$ and $m^{***}$ need not coincide in general (cf. \cite{Arens51}). The mapping $m$ is called \emph{Arens regular} if $m^{t***t}=m^{***}$. One of the best known examples of Arens regular maps is given by the product of any C$^*$-algebra. That is, every C$^*$-algebra $A$ is Arens regular and the unique Arens extension of the product of $A$ to $A^{**}\times A^{**}$ coincides with the product of its enveloping von Neumann algebra (cf. \cite[Corollary 3.2.37]{Dales00}).\smallskip

It is worth to recall some notions. Two projections $p$ and $q$ in a von Neumann algebra $M$ are called \emph{(Murray-von Neumann) equivalent} if there is a partial isometry $e\in M$ such that $e^* e = p$ and $ee^*=q$. We write this fact as $p\sim q$. A projection $p$ in $M$ is said to be \emph{finite} if there is no projection $q< p$ that is equivalent to $p$. A projection $p$ in $M$ is {\em infinite} if it is not finite, and {\em properly infinite} if $p\ne 0$ and $zp$ is infinite whenever $z$ is a central projection such that $zp\ne0$ (cf. \cite[Definition V.1.15]{Tak}). The von Neumann algebra $M$ is said to be finite, infinite, or properly infinite according to the property of its identity \cite[Definition V.1.16]{Tak}.\smallskip

Let $M$ be a von Neumann algebra. A (\emph{faithful}) \emph{center-valued trace} on $M$ is a linear mapping $\tau$ from $M$ onto its center $Z(M)$ satisfying:
\begin{enumerate}[$(a)$]\item $\tau(x^* x) = \tau(x x^*) \geq 0$ for all $x\in M$;
\item $\tau(a x) = a \tau(x),$ for all $a\in Z(M)$, $x\in M$;
\item $\tau(1) =1$;
\item $\tau(x^* x) \neq 0$ for every nonzero $x\in M$.
\end{enumerate} A von Neumann algebra $M$ is finite if and only if it admits a faithful center-valued trace (which is further weak$^*$-continuous) \cite[Theorem V.2.6]{Tak}.\smallskip

Suppose $X$ is a Banach $A$-bimodule over a C$^*$-algebra $A.$ Let $\pi_1: A\times X \to X$ and $\pi_2: X\times A \to X$ stand for the corresponding module operations given by $\pi_1(a,x) = a x$ and $\pi_2(x,a) = x a$, respectively. Given $a\in A^{**}$ and $z\in X^{**},$ we shall write $a z = \pi_1^{***} (a,z)$ and $z a = \pi_2^{***} (z,a)$. It is known that $X^{**}$ is a Banach $A^{**}$-bimodule (and also a Banach $A$-bimodule) for the just defined operations (\cite[Theorem 2.6.15$(iii)$]{Dales00}). An additional property of this construction tells that \begin{align}\label{eq product bidual module} a x &= \pi_1^{***}(a,x) = w^*\hbox{-}\lim_{\lambda} w^*\hbox{-}\lim_{\mu} a_\lambda x_\mu, \hbox{ and }\\ x a &= \pi_2^{***} (x,a) =  w^*\hbox{-}\lim_{\mu} w^*\hbox{-}\lim_{\lambda}  x_\mu a_\lambda, \nonumber\end{align}\label{eq weak* continuity properties module product on bidual}in the weak$^*$ topology of $X^{**},$ whenever $(a_\lambda)$ and $(x_\mu)$ are nets in $A$ and $X$, respectively, such that $a_\lambda \to a\in A^{**}$ in the weak$^*$ topology of $A^{**}$ and $x_\mu\to x\in X^{**}$ in the weak$^*$ topology of $X^{**}$ (cf. \cite[$(2.6.26)$]{Dales00}). The reader should be warned that the module operations on $X^{**}$ need not be separately weak$^*$ continuous. This handicap produces some difficulties in our arguments.\smallskip

Let $d: A\to M$ be an (anti-)derivation from a C$^*$-algebra into a Banach $A$-bimodule. We have already seen that $d$ is continuous and hence $d^{**} : A^{**} \to X^{**}$ is weak$^*$-continuous. It follows from \eqref{eq product bidual module} that $d^{**}$ is (anti-)derivation.\smallskip

It is clear from the above that in a C$^*$-algebra where every element coincides with a finite sum of commutators every anti-derivation is zero. In particular, every anti-derivation from a properly infinite von Neumann algebra $M$ into a Banach $M$-bimodule is zero by Halpern's theorem  \cite[Theorem 3.10]{Halpern69}. We can improve this conclusion in the next result.

\begin{thm}\label{t Cstar generated by commutators and all} Let $\delta: A\to X$ be an anti-derivation from a C$^*$-algebra to a Banach $A$-bimodule.
Then there exists a finite central projection $p_1$ in $A^{**}$ and an element $x_0\in X^{**}$ such that $$\delta^{**} (x) = \delta_{x_0} (\tau(p_1 x)) = [\tau(p_1 x) , x_0], \hbox{ for all } x\in A^{**},$$ where $\tau: p_1 A^{**}\to Z(p_1 A^{**})$ is the (faithful) center-valued trace on $p_1 A^{**}$.\smallskip

Furthermore, if we also assume that $ z x = x z$ for all $z\in Z(A^{**})$ and $x\in X^{**}$ {\rm(}for example when $X = A$ or $A^*$ or $A^{**}${\rm)}, every anti-derivation $\delta: A\to X$ is zero.
\end{thm}

\begin{proof} We have already commented in previous paragraphs that $\delta$ is a continuous derivation with $\delta([a,b])=0$ for all $a,b\in A$. Furthermore $\delta^{**} : A^{**}\to X^{**}$ is an anti-derivation too. Therefore $\delta^{**}$ is a derivation and vanishes on every commutator of $A^{**}$.\smallskip

By \cite[Theorem V.1.19 and Lemma V.1.7]{Tak} the identity of the von Neumann algebra $A^{**}$ is uniquely written as the sum of (centrally) orthogonal projections $p_1$ and $p_2$ such that $p_1$ is finite and $p_2$ is properly infinite. It follows that $A^{**}$ decomposes as the orthogonal direct sum of $M_1 = p_1 A^{**} p_1$ and $M_2 = p_2 A^{**} p_2$, $M_1$ is finite and $M_2$ is properly infinite. By Theorem 3.10 in \cite{Halpern69} every element $a_2$ in $M_2$ can be written as the sum of 2 commutators in $M_2$. Furthermore, let $\tau: M_1\to Z(M_1)$ be the faithful center-valued trace of the finite von Neumann algebra $M_1$. Theoreme 3.2 in \cite{FackdelaHarpe80} asserts that every element $b_1$ in $M_1$ with $\tau(b_1)=0$ coincides with the sum of 10 commutators in $M_1$. Since for each $a_1\in M_1$ the element $b_1 = a_1 - \tau(a_1)\in M_1$ has zero trace, we deduce that $a_1 - \tau(a_1)$ writes as the sum of 10 commutators in $M_1$. Since commutators in $A^{**}$ are sums of commutators in $M_1$ and $M_2$, we deduce from the properties of $\delta^{**}$ that $$\delta^{**} (a_1+a_2) = \delta^{**} (\tau(a_1)), \hbox{ for all } a_1\in M_1, a_2\in M_2.$$

On the other hand $X^{**}$ is a dual $Z(M_1)$-bimodule with respect to the restricted bimodule operations, and $\delta^{**}|_{Z(M_1)}: Z(M_1)\to X^{**}$ is a derivation and an anti-derivation. Since every commutative C$^*$-algebra $\mathcal{C}$ is amenable (i.e., for every Banach $\mathcal{C}$-bimodule $Y$, every derivation from $\mathcal{C}$ into $Y^*$ is inner (see, for example, \cite[Theorem 5.6.2$(i)$]{Dales00}), there exists $x_0\in X^{**}$ such that $\delta^{**} (z) = [z,x_0]$ for all $z\in Z(M_1)$. This finishes the proof of the first statement.\smallskip

Let us deal with the last statement. Suppose that $ z x = x z$ for all $z\in Z(A^{**})$ and $x\in X^{**}$. In this case $\delta^{**} (z) =[z,x_0]=0$ for all $z\in Z(M_1),$ and thus $\delta=0$.
\end{proof}

Let $A=C[0,1]$ and $X= \mathbb{C}$ equipped with the bimodule operations defined by $f\cdot \lambda := f(0) \lambda$ and $ \lambda \cdot f := f(1) \lambda$. According to this structure, there exist elements $f\in C[0,1]$ such that $f\cdot 1 \neq 1\cdot f$. The mapping $\delta_{1} : C[0,1]\to X$ is a non-zero anti-derivation.

\section{Linear maps anti-derivable at zero}\label{sec: continuous linear maps which are anti-derivable at zero}

Let $A$ and $B$ be C$^*$-algebras. It is known that every bounded bilinear form $V: A\times B\to \mathbb{C}$ admits a unique norm preserving separately weak$^*$ continuous extension to $A^{**}\times B^{**}$ (cf. \cite[Lemma 2.1]{JohnKadRing72}). Actually the same conclusion also holds when $A$ and $B$ are JB$^*$-triples (see \cite[Lemma 1]{PeRo2001}).\smallskip

Along this note, the self-adjoint part of a C$^*$-algebra $A$ will be denoted by $A_{sa}$.\smallskip

Let $A$ be a Banach algebra. If instead of requiring $A$ to be zero product determined we only request that for every Banach space $Y$ and every continuous bilinear form $V : A\times A\to Y$ satisfying $V (a,b) =0$ for every $a,b\in A_{sa}$ with $a b =0$, there exist continuous functionals $\phi,\varphi\in  A^*$ such that $V(a, b) = \phi(a b) + \varphi (b a)$ for all $a,b\in A$, a celebrated theorem due to Goldstein (see \cite[Theorem 1.10]{Gold}) affirms that every C$^*$-algebra satisfies this latter property. This is one of the advantages in the study of derivations on C$^*$-algebras.\smallskip

Let $X$ be a Banach $A$-bimodule over a Banach algebra $A$. According to the usual terminology, we shall say that $X$ is \emph{essential} if the linear span of the set $\{ a x b: a,b\in A, x\in X\}$ is dense in $X$.\smallskip

If $A$ is a non-unital C$^*$-algebra, $1$ denotes the unit in $A^{**}$ and  $(u_{\lambda})$ is a bounded approximate unit in $A$ (cf. \cite[Theorem 1.4.2]{Ped}), it is known that $(u_{\lambda})\to 1$ in the weak$^*$-topology of $A^{**}$. Furthermore, if we regard $X^{**}$ as a Banach $A^{**}$-bimodule, it follows from the basic properties commented in the first section that $$(\eta a) 1 = w^*\hbox{-}\lim_{\lambda} (\eta a) u_{\lambda} = w^*\hbox{-}\lim_{\lambda} \eta (a u_{\lambda}) = \|.\|\hbox{-}\lim_{\lambda} \eta (a u_{\lambda}) = \eta a,$$ for all $a\in A$ and $\eta \in X$. Assuming that $X$ is essential we get $\eta 1 = \pi_2^{***} (\eta,1) = \eta$ (and similarly $1 \eta = \pi_1^{***} (1,\eta) = \eta$) for all $\eta \in X$. Actually $\lim_{\lambda} \|\eta u_{\lambda} -\eta\| = 0 = \lim_{\lambda} \| u_{\lambda} \eta -\eta\|$ for all $\eta\in X$. Let us take $\eta \in X^{**}$, and pick via Goldstine's theorem a bounded net $(\eta_{\mu})$ in $X$ converging to $\eta$ in the weak$^*$ topology of $X^{**}$. Since $\pi_2^{***}(\cdot, 1)$ is weak$^*$ continuous we have \begin{equation}\label{eq unit on the right bidual module} \eta 1 =\pi_2^{***} (\eta,1) = w^*\hbox{-}\lim_{\mu} \pi_2^{***} (\eta_{\mu},1) = w^*\hbox{-}\lim_{\mu} \eta_{\mu} 1 = w^*\hbox{-}\lim_{\mu} \eta_{\mu} = \eta.
\end{equation}

Our next result is a modular version of \cite[Lemma 2.2]{FaaGhahra2019}.

\begin{lemma}\label{l centrality in modules} Let $X$ be a Banach $A$-bimodule over a C$^*$-algebra $A$. Let $\xi$ be an element in $X$ satisfying the following property: $h \xi k=0$ for every  $h,k\in A_{sa}$ with $h k=0$. Let $1$ denote the unit of $A^{**}$. Then the element $\eta =  1 \xi 1\in X^{**}$ satisfies $a \eta b = a \xi b,$ for all $a,b\in A,$ and commutes with every element in $A$, that is, $\eta a = a \eta$, for all $a\in A$.
\end{lemma}

\begin{proof} Since $A$ may be non-unital we shall consider $A^{**}$ and the space $X^{**}$ as a Banach $A^{**}$-bimodule. Let $\eta = 1 \xi 1 \in X^{**}$. By the basic properties of the $A^{**}$-bimodule $X^{**}$, we have $a \eta b = a \xi b$ for all $a,b\in A$ and $\eta 1 = 1 \eta = \eta$.  \smallskip

Let us fix an arbitrary $\phi\in X^*$ and define the bounded bilinear form given by $V_{\phi} : A\times A\to \mathbb{C}$, $V_{\phi} (a,b) = \phi (a \eta b)= \phi (a \xi b)$. It follows from the hypothesis that $V_{\phi} (h,k)= 0,$ for every  $h,k\in A_{sa}$ with $h k=0$, witnessing that $V_{\phi}$ is an orthogonal form in the sense of Goldstein \cite{Gold}. Theorem 1.9 in \cite{Gold} implies the existence of two functionals $\varphi_1,\varphi_2\in A^*$ satisfying $V_{\phi} (a,b) = \varphi_1(a b) + \varphi_2 (b a)$ for all $a,b\in A$. We denote by the same symbol $V_{\phi}$ the (unique) separate weak$^*$ continuous extension of $V_{\phi}$ to $A^{**}\times A^{**}$. We can therefore conclude that $$\phi( a \eta) =  V_{\phi} (a,1) =\varphi_1(a) + \varphi_2 (a) = V_{\phi} (1,a) = \phi (\eta a), $$ for all $a\in A.$ The arbitrariness of $\phi\in X^*$ combined with the Hahn-Banach theorem implies that $\eta a = a \eta$, for all $a\in A$.
\end{proof}

We can now present our characterization of those continuous linear maps on a C$^*$-algebra which are anti-derivable at zero.

\begin{thm}\label{t characterization of anti-derivations} Let $T: A\to X$ be a bounded linear operator where $A$ is a C$^*$-algebra and $X$ is an essential Banach $A$-bimodule. Then the following are equivalent:\begin{enumerate}[$(a)$]\item $T$ is anti-derivable at zero;
\item There exists an anti-derivation $d:A\to X^{**}$ and an element $\xi \in X^{**}$ satisfying $\xi a = a \xi,$ $\xi [a,b]=0,$ $T(a b) = b T(a) + T(b) a - b \xi a,$ and $T(a) = d(a) + \xi a,$ for all $a,b\in A$;
\end{enumerate}
\end{thm}

\begin{proof} $(a)\Rightarrow (b)$ Suppose $T$ is anti-derivable at zero. Let us pick $h_1,k,h_2\in A_{sa}$ with $h_j k =0$ (and thus $k h_j =0$) for $j=1,2$, it follows from the hypothesis that $T(k) h_2 + k T(h_2) =0$ and therefore \begin{equation}\label{eq conditions for gen der} h_1 T(k) h_2 = h_1( T(k) h_2 + k T(h_2)) =0.
 \end{equation} This shows that the mapping $T: A\to X$ satisfies the hypotheses of \cite[Theorem 2.11]{AyuKudPe2014}, we therefore conclude from the just quoted result that $T:A\to X$ is a generalized derivation, that is, there exists  $\xi\in X^{**}$ such that \begin{equation}\label{eq T is a generalized derivation}T( a b ) = T(a) b + a T(b) - a \xi b, \ \ \forall a,b\in A.
\end{equation} By replacing $\xi$ with $1 \xi 1$ we can assume that $1 \xi = \xi 1= \xi$ and \eqref{eq T is a generalized derivation} holds.\smallskip

It is not hard to check from \eqref{eq T is a generalized derivation} that the mapping $d: A\to X^{**}$, $d(a) = T(a) -\xi a$ is a derivation satisfying $T(a) = d(a) + \xi a$ for all $a\in A$.\smallskip

If we pick $h,k\in A_{sa}$ with $h k =0$ (and thus $k h  =0$). We deduce from the hypothesis that $T(h) k + h T(k) =0,$ and by \eqref{eq T is a generalized derivation} $$0= T(h k ) = T(h) k + h T(k) - h \xi k ,$$ identities which combined give $h \xi k=0$ (for any $h,k\in A_{sa}$ with $h k =0$). Lemma \ref{l centrality in modules} guarantees that $\xi a = a \xi$ for all $a\in A$. 
\smallskip

We shall next show that $d$ is an anti-derivation. Let $(Y, \odot)$ denote the opposite Banach $A$-bimodule $X^{op}$, that is, $y \odot a = a y$ and $a \odot y = ya$ for all $a\in A$, $y\in Y$. Let us pick $h_1,k,h_2\in A_{sa}$ with $h_j k =0$ for $j=1,2$. We have seen in \eqref{eq T is a generalized derivation} that $ h_1 \odot T(k) \odot h_2 = h_2 T(k) h_1 =0.$ Then the mapping $\tilde{T}: A\to Y$, $\tilde{T}(a) =T(a)$ ($a\in A$) satisfies that  $ h_1 \odot \tilde{T}(k) \odot h_2 =0$ for every $h_1,k,h_2\in A_{sa}$ with $h_j k =0$. We deduce from \cite[Theorem 2.11]{AyuKudPe2014} the existence of $\eta\in X^{**}$ such that $T (a b ) = T(a) \odot b + a\odot T(b) - a\odot \eta \odot b, \ \ \forall a,b\in A,$ equivalently, \begin{equation}\label{eq second step toward an anti-derviation} T( a b  ) = b T(a) + T(b) a - b \eta a, \ \ \forall a,b\in A.
\end{equation} Replacing $\eta$ with $1 \eta 1$ we can always assume that $\eta = \eta 1 = 1 \eta$.\smallskip

By mimicking the arguments above, fix $h,k\in A_{sa}$ with $h k =0$. We deduce from the hypothesis (with $k h =0$) that $T(h) k + h T(k) =0,$ and by \eqref{eq second step toward an anti-derviation} $$0= T(k h) = T(h) k + h T(k) - h \eta k .$$ By combining the previous two identities we get $h \eta k=0$ for all $h,k\in A_{sa}$ with $h k =0$. A new application of  Lemma \ref{l centrality in modules} guarantees that $\eta a = a \eta$ for all $a\in A$.\smallskip

Now, combining the fact that $\xi$ and $\eta$ commute with any element in $A$, and \eqref{eq T is a generalized derivation} and \eqref{eq second step toward an anti-derviation} with $a=b$, we have $\eta a^2 = \xi a^2 $ for all $a\in A.$ Since $A$ is a C$^*$-algebra it follows that $a \eta = \eta a = \xi a = a \xi$ for all $a\in A.$ Therefore, there is no loss of generality in assuming $\xi = \eta$ in \eqref{eq T is a generalized derivation} and \eqref{eq second step toward an anti-derviation}.\smallskip

Now, let us apply \eqref{eq second step toward an anti-derviation} and \eqref{eq T is a generalized derivation} to deduce the following identities
$$ T(ab) = T(b) a + b T(a) - b\xi a, \hbox{ and } T(ba) =T(b) a + b T(a) - b \xi a,$$ for all $a,b\in A$. Therefore \begin{equation}\label{eq T annihilates on commutators} T([a,b])= T(ab-ba) = 0 \hbox{ for all } a ,b\in A.
\end{equation}

Let us analyze the identity \eqref{eq T is a generalized derivation}. Let $(u_{\lambda})$ be an approximate unit in $A$. Since the identity
$$T( u_{\lambda} b ) = T(u_{\lambda}) b + u_{\lambda} T(b) - u_{\lambda} \xi b$$ holds for every $\lambda$, $T^{**}$ is weak$^*$ continuous (and hence $T^{**}(u_{\lambda})\to T^{**}(1)$ in the weak$^*$ topology),  $u_{\lambda} T(b)\to T(b)$ in norm because $X$ is essential, the product of $A^{**}$ is separately weak$^*$ continuous \cite[Theorem 1.7.8]{Sa}, $\pi_1^{***} (\cdot, \xi b)$ is weak$^*$ continuous (and thus $u_{\lambda} \xi b =\pi_1^{***} (u_{\lambda}, \xi b)\to \pi_1^{***} (1, \xi b) = 1 \xi b=\xi b$ in the weak$^*$ topology) we conclude that \begin{equation}\label{eq T**(1) is almost xi} T( b ) = T^{**}(1) b + T(b) - 1 \xi b, \hbox{ or equivalently, } T^{**}(1) b = 1 \xi b=\xi b,
\end{equation} for all $b\in A$.\smallskip

Let us recall that a continuous bilinear mapping $V: A\times A\to X$ \emph{preserves zero products} if $$ab=0 \hbox{ in } A \Rightarrow V(a,b) =0.$$ By \cite[Example 1.3$(2.)$, Theorem 2.11 and Definition 2.2]{AlBreExVill09} every continuous bilinear mapping $V$ preserving zero products satisfies $V(ab,c) = V(a, bc)$ for all $a,b,c\in A$. By hypothesis, the mapping $V(a,b) := T(b) a + b T(a)$ is continuous and preserves zero products, therefore $$  T(c) ab + c T(ab) = V(ab,c) = V(a, bc) = T(b c) a + bc T(a),$$ for all $a,b,c\in A$. If in the above equality we replace $c$ with $u_{\lambda}$, where $(u_{\lambda})$ is an approximate unit in $A$, and we take weak$^*$ limits we get $$  T^{**}(1) ab + T(ab) = T(b) a + b T(a), \hbox{ for all $a,b\in A$.}$$
Since $\xi ab =  T^{**}(1) ab$ (cf. \eqref{eq T**(1) is almost xi}) and $T(ab) = T(ba) = d(ba) + \xi ba$ it follows that
$$\xi ab + d(ba) + \xi ba = d(b) a + \xi ba + b d(a) + \xi ba  = d(ba) + 2 \xi ba ,$$ witnessing that $\xi [a,b]=0$.\smallskip

Therefore, by \eqref{eq T annihilates on commutators}
$$ d([a,b]) =  T([a,b]) - \xi [a,b] = 0 \hbox{ for all $a,b\in A$}.$$ Lemma \ref{l characterization anti-derivations} proves that $d$ is an anti-derivation.\smallskip

%

$(b)\Rightarrow (a)$ Suppose there exist an anti-derivation $d:A\to X^{**}$ and an element $\xi \in X^{**}$ satisfying the stated properties. Let us take $a,b\in A$ with $a b=0$. It follows from the assumptions that $$T(b) a + b T(a) = (d(b) + \xi b) a + b (d(a) + \xi a) = d(a b) + 2 \xi ba = 0 + 2 \xi ab = 0.$$
\end{proof}

\begin{remark}\label{remark extra conclusions for main theorem 1} If in Theorem \ref{t characterization of anti-derivations} the C$^*$-algebra $A$ is unital or if $X$ is a dual Banach $A$-bimodule statement $(b)$ can be replaced with the following:
\begin{enumerate}[$(b^\prime)$]
\item There exist an anti-derivation $d:A\to X$ and an element $\xi \in X$ satisfying $\xi a = a \xi,$ $\xi [a,b]=0,$ $T(a b) = b T(a) + T(b) a - b \xi a,$ and $T(a) = d(a) + \xi a,$ for all $a,b\in A$.
\end{enumerate}
A closer look at the proof of Theorem \ref{t characterization of anti-derivations} shows that the desired statement will follow as soon as we prove that the element $\xi$ lies in $X$. If $A$ is unital this is clear because $T(1) = d(1) + \xi = \xi \in X.$ If $X$ is a dual Banach space, we can repeat the arguments in the proof of \cite[Proposition 4.3 or Theorem 4.6]{AlBreExVill09}.
\end{remark}

Every C$^*$-algebra $A$ is an essential $A$-bimodule because it admits a bounded approximate unit (see \cite[Theorem 1.4.2]{Ped}). The second dual, $A^{**}$, of $A$ is an $A$-bimodule with respect to the natural product. In general, $A^{**}$ need not be an essential $A$-bimodule, consider, for example, $A= c_0$ and $A^{**} = \ell_{\infty}$. However, if $A$ is unital, $A^{**}$ is an essential $A$-bimodule. Despite that $A^{**}$ is not in general an essential $A$-bimodule, $A$ is weak$^*$ dense in $A^{**}$ by Goldstine's theorem, and it is known that $A$ admits a bounded approximate unit (see \cite[Theorem 1.4.2]{Ped}) which converges to the unit of $A^{**}$ in the weak$^*$ topology. Applying these special properties the proofs of \cite[Lemma 2.10, Theorem 2.11]{AyuKudPe2014} remain valid to characterize when a bounded linear operator $T: A\to A^{**}$ is a generalized derivation. Therefore the proof of Theorem \ref{t characterization of anti-derivations} above can be combined with Theorem \ref{t Cstar generated by commutators and all} to get the following result.

\begin{thm}\label{t characterization of anti-derivations bidual} Let $T: A\to A^{**}$ be a bounded linear operator where $A$ is a C$^*$-algebra. Then the following are equivalent:\begin{enumerate}[$(a)$]\item $T$ is anti-derivable at zero;
\item There exists $\xi \in A^{**}$ satisfying $\xi a = a \xi,$ $\xi [a,b]=0,$  and $T(a) = \xi a$, for all $a,b\in A$.
\end{enumerate}
\end{thm}

The preceding theorem can be regarded as a generalization of \cite[Corollary 3.8$(i)$]{GhaPan2018}.

\section{Linear maps $^*$-anti-derivable at zero}\label{sec: continuous linear maps which are *-anti-derivable at zero}

In this section we shall deal with continuous linear maps which are r-$^*$-anti-derivable at zero. We shall first recall the basic theory on bimodules equipped with an involution. Let $X$ be a Banach $A$-bimodule over a C$^*$-algebra $A$. By an \emph{$A$-bimodule involution} on $X$ we mean a continuous conjugate linear mapping $X\to X$, $x \mapsto x^*,$ satisfying $(x^*)^*=x,$ $(a x)^* = x^* a^*$ and $(x a)^* = a^* x^*$, for all $a\in A$, $x\in X$. The natural involutions on $A$ and on $A^{**}$ are $A$-bimodule involutions when $A$ and $A^{**}$ are regarded as Banach $A$-bimodules. Another typical example can be given in the following way: for each functional $\varphi\in A^*$ and $a \in A$, the functionals $a\varphi,\varphi a\in A^*$ are defined by $(a\varphi) (b)= \varphi (b a),$ and $(\varphi a) (b) =\varphi (a b)$, for all $b\in A$, respectively. These operations define a structure of Banach $A$-bimodule on $A^*$. Furthermore, for each $\varphi\in A^*$ we define $\varphi^*\in A^*$ by $\varphi^* (b) := \overline{\varphi(b^*)}$ ($\forall b\in A$). It is easy to check that $(a \varphi)^* = \varphi^* a^*$ and $(\varphi a)^* =a^* \varphi^*$ for all $a\in A$, $\varphi\in A^*$. Therefore $\varphi\mapsto \varphi^*$ defines an $A$-bimodule involution on $A^*$.\smallskip

Suppose $x\mapsto x^*$ is an $A$-bimodule involution on a Banach $A$-bimodule $X$. We shall regard $X^*$ as a Banach $A$-bimodule with module operations given by $(a\phi) (x)= \phi (x a),$ and $(\phi a) (x) =\phi (a x)$, for all $a\in A$, $x\in X$ and $\phi\in X^*$. We shall consider the natural involutions on $X$ and $X^{**}$ naturally induced by the $A$-bimodule involution of $X,$ defined by $\phi^* (x) := \overline{\phi(x^*)}$ ($\forall \phi\in X^*, x\in X$) and $z^* (\phi) :=\overline{z(\phi^*)}$ ($\forall \phi\in X^*, z\in X^{**}$). Clearly, the involution $z\mapsto z^*$ is weak$^*$ continuous on $X^{**}$. Let $a\in A,$ $z\in X^{**},$ and let $(x_{\mu})\subset X$ a bounded net converging to $z$ in the weak$^*$ topology of $X^{**}$.  By the properties of the module operation on $X^{**}$ (see page \pageref{eq weak* continuity properties module product on bidual}) we have $$(a z)^* = \pi_1^{***} (a,z)^* = w^*{\hbox{-}}\lim_{\mu} \pi_1 (a,x_{\mu})^* =  w^*{\hbox{-}}\lim_{\mu} \pi_2 (x_{\mu}^*,a^*) = \pi_2^{***} (z^*,a^*) = z^* a^*,$$ and
$$(z a)^* = \pi_2^{***} (z,a)^* = w^*{\hbox{-}}\lim_{\mu} \pi_2 (x_{\mu},a)^* =  w^*{\hbox{-}}\lim_{\mu} \pi_1 (a^*,x_{\mu}^*) = \pi_1^{***} (a^*,z^*) = a^* z^*.$$

A Banach $A$-bimodule equipped with an $A$-bimodule involution will be called a \emph{Banach $^*$-$A$-bimodule}. Along this section $X$ will stand for a Banach $^*$-$A$-bimodule over a C$^*$-algebra $A$. A linear mapping $T :A\to X$ will be called \emph{r-$^*$-anti-derivable at zero} (respectively,  \emph{l-$^*$-anti-derivable at zero}) if $ab^* = 0$ in $A$ implies $T(b)^* a + b^* T(a) = 0$ in $X$ (respectively, $a^* b = 0 \hbox{ in } A \Rightarrow T(b) a^* + b T(a)^* = 0$ in $X$). It is easy to see that $T$ is r-$^*$-anti-derivable at zero if and only if the mapping $S : A\to X$, $S(a) :=T(a^*)^*$ ($\forall a\in A$) is l-$^*$-anti-derivable at zero. \smallskip

We can now state our main conclusion for continuous linear maps which are r-$^*$-anti-derivable at zero.

\begin{thm}\label{t maps r-*-anti-derivable at zero} Let $T: A\to X$ be a bounded linear operator where $A$ is a C$^*$-algebra and $X$ is an essential Banach $^*$-$A$-bimodule. Then the following statements are equivalent:\begin{enumerate}[$(a)$]\item $T$ is r-$^*$-anti-derivable at zero {\rm(}i.e., $ab^* = 0 \hbox{ in } A \Rightarrow T(b)^* a + b^* T(a) = 0$ in $X${\rm)};
\item There exists a $^*$-derivation $d:A\to X^{**}$ and an element $\xi \in X^{**}$ satisfying the following properties: \begin{enumerate}[$(i)$]\item  $d([a,b]) + [a,b] \xi + \xi^* [a,b]=0,$ for all $a,b\in A$;
\item $T(a b) = a T(b) + T(a) b - a \xi b,$ and $T(a) = d(a) + a \xi $ for all $a,b\in A$.
\end{enumerate}
\end{enumerate}

If we further assume that $A$ is unital or $X$ is a dual Banach $A$-bimodule, we can replace $(b)$ with the following:
\begin{enumerate}[$(b')$]\item There exists a $^*$-derivation $d:A\to X$ and an element $\xi \in X$ satisfying the properties $(i)$-$(ii)$ above.
\end{enumerate}
\end{thm}

\begin{proof} $(a)\Rightarrow (b)$ Suppose $T$ is r-$^*$-anti-derivable at zero. As in the proof of Theorem \ref{t characterization of anti-derivations} we observe that given $h_1,h_2,k\in A_{sa}$ with $h_j k =0$ for $j=1,2$ we have $T(h_1)^* k + h_1 T(k) =0$, and consequently, \begin{equation}\label{eq 0 T is a gener derivation} 0=(T(h_1)^* k + h_1 T(k)) h_2 =h_1 T(k) h_2.
 \end{equation} Theorem 2.11 in \cite{AyuKudPe2014} assures that $T$ is a generalized derivation, that is, there exists an element $\xi\in X^{**}$ satisfying \begin{equation}\label{eq T *-antiderviable is a gen derivation} T(a b) = T(a) b + a T(b) - a \xi b, \ \ \forall a,b\in A.
\end{equation} By replacing $\xi$ with $1\xi 1$ we can assume that $\xi = 1 \xi = \xi 1$.  It is routine to check that the mapping $d: A\to X^{**}$, $d(a) = T(a) - a \xi $ is a derivation and $T(a) = d(a) + a \xi $ for all $a\in A$.\smallskip

The same arguments employed in the proof of Theorem \ref{t characterization of anti-derivations} \eqref{eq T**(1) is almost xi} prove that \begin{equation}\label{eq T**(1) is almost xi second} T^{**}(1) b = 1 \xi b=\xi b, \hbox{  for all $b\in A$.}
\end{equation}

We consider the continuous bilinear mapping $V:A\times A\to X$ defined by $V(a,b) := T(b^*)^* a + b T(a).$ If $a (b^*)^* = a b = 0$ the hypothesis implies that $V(a,b)=0$. Therefore $V$ preserves zero products. We conclude from \cite[Example 1.3$(2.)$, Theorem 2.11 and Definition 2.2]{AlBreExVill09} that $V(ab,c) = V(a,bc),$ or equivalently $$ T(c^*)^* ab + c T(ab) = T(c^* b^*)^* a + b c T(a) $$ for all $a,b,c\in A$. So, if $(u_{\lambda})$ is an approximate unit in $A$ we have $$ T(u_{\lambda})^* ab + u_{\lambda} T(ab) = T(u_{\lambda} b^*)^* a + b u_{\lambda} T(a),$$ for all $\lambda$, $a,b\in A$. We can take norm limits on the right hand side. For the left hand side we observe that the bimodule operations $\pi_2^{***}(\cdot, ab)$ and $\pi_1^{***}(\cdot,T(ab))$ are weak$^*$ continuous. Taking weak$^*$ limits in $\lambda$ in the previous equality we derive \begin{equation}\label{eq 4.4.0} T^{**}(1)^* ab + T(ab) =T^{**}(1)^* ab + 1 T(ab) = T( b^*)^* a + b  T(a),
 \end{equation} and \begin{equation}\label{eq one 2210} b^* a^* T^{**}(1) + T(ab)^* = a^* T( b^*) +T(a)^* b^*,
 \end{equation} for all $a,b\in A$. Replacing $a$ with $u_{\lambda}$ we get $$b^* u_{\lambda} T^{**}(1) + T( u_{\lambda} b)^* = u_{\lambda} T( b^*) +T(u_{\lambda})^* b^* ,$$ for all $\lambda$, $b\in A$. Now, taking weak$^*$ limits in $\lambda$ we obtain $$b^* T^{**}(1) + T( b)^* = T( b^*) +T^{**}(1)^* b^* ,$$ equivalently, $$  T( b)^* - T^{**}(1)^* b^*  = T( b^*) - b^* T^{**}(1), \hbox{ for all } b\in A.$$

Consequently, \begin{equation}\label{eq new 2210} T(h) + T^{**}(1)^* h = T(h)^* + h T^{**}(1), \ \ \forall h\in A_{sa},
\end{equation} that is,
$$ d(h) + h \xi +T^{**}(1)^* h   = d(h)^* + \xi^* h  + h T^{**}(1), \ \ \forall h\in A_{sa}.$$ Multiplying by $k_1\in A_{sa}$ on the left and by $k_2\in A_{sa}$ on the right and applying \eqref{eq T**(1) is almost xi second} we get $$\begin{aligned} k_1 d(h) k_2 + k_1 h \xi k_2 + k_1 \xi^* h k_2  &= k_1 d(h) k_2 + k_1 h \xi k_2 + k_1 T^{**}(1)^* h k_2  \\
= k_1 d(h)^* k_2 + k_1 \xi^* h k_2 + k_1 h T^{**}(1) k_2 &= k_1 d(h)^* k_2 + k_1 \xi^* h k_2 + k_1 h \xi k_2,
\end{aligned}$$ for all $h,k_1,k_2\in A_{sa}$. It then follows that $$ a d(h) b = a d(h)^* b, \hbox{ for all } h\in A_{sa}, a,b\in A.$$ Since the mappings $\pi_1^{***} (\cdot, d(h) b )$ and $\pi_1^{***} (\cdot, d(h)^* b)$ are weak$^*$ continuous, we can replace $a$ with $u_{\lambda}$ and take weak$^*$ limits in $\lambda$ to deduce that \begin{equation}\label{eq d(h) and d(h)* coincide up to products on the right} d(h) b = d(h)^* b, \hbox{ and } b^* d(h)^* = b^* d(h), \hbox{ for all } h\in A_{sa}, b\in A.
\end{equation}

Now, by the local Gelfand theory, for each $h\in A_{sa}$, there exist $h_1,h_2\in A_{sa}$ with $h_1 h_2 = h_2 h_1 = h$. If we apply \eqref{eq d(h) and d(h)* coincide up to products on the right} and the fact that $d$ is a derivation we arrive at $$d(h)^* = d(h_1 h_2)^*= \left(d(h_1) h_2 + h_1 d(h_2) \right)^* = h_2 d(h_1)^* + d(h_2)^*  h_1$$
$$ = h_2 d(h_1) + d(h_2)  h_1 = d (h_2 h_1) = d(h).$$ The arbitrariness of $h\in A_{sa}$ proves that $d(a)^* = d(a^*)$ for all $a\in A$, witnessing that $d$ is a $^*$-derivation.\smallskip

We claim that \begin{equation}\label{eq d plus xi* on the left annihilates on Lie products} d([a,b]) + \xi^* [a,b] + [a,b]  \xi  = 0 \hbox{ for all } a,b\in A.
\end{equation} Namely, by \eqref{eq T**(1) is almost xi second} and \eqref{eq 4.4.0} for all $a,b,c\in A$ we have $$ c \xi^* a b +c  T(ab) = c T^{**}(1)^* a b +c  T(ab) = c T( b^*)^* a + c b T(a),$$ equivalently $$ c \xi^* a b +c  d(ab) + c ab \xi  = c d(b^*)^* a + c \xi^* b a + c b d(a) + c ba \xi = c d(ba) + c \xi^* b a + c ba \xi,$$ where in the last equality we applied that $d$ is $^*$-derivation. Therefore $$ c d([a,b]) + c \xi^* [a,b] + c[a,b]\xi =0, \hbox{ for all } a,b,c \in A.$$ Replacing $c$ with $u_{\lambda}$, where $(u_{\lambda})$ is an approximate unit in $A$, and having in mind that the maps $\pi_1^{***} (\cdot, d([a,b]))$ and $\pi_1^{***} (\cdot, \xi^* [a,b])$
are weak$^*$ continuous, by taking weak$^*$-limits in $\lambda$ we obtain the identity claimed in \eqref{eq d plus xi* on the left annihilates on Lie products}.\smallskip

$(b)\Rightarrow (a)$ Suppose there exists a $^*$-derivation $d:A\to X^{**}$ and an element $\xi \in X^{**}$ satisfying properties $(i)$-$(ii)$ in the statement. Let us fix $a,b\in A$ with $a b^*=0$. It follows from the assumptions that $$ T(b)^* a + b^* T(a) = d(b)^* a + \xi^* b^* a + b^* d(a) + b^* a \xi = d(b^*) a  + b^* d(a) + \xi^* b^* a  + b^* a \xi$$ $$ = d(b^* a) + \xi^* b^* a + b^* a \xi 
= d([b^*, a]) + \xi^* [b^*, a] + [b^*, a] \xi  =\hbox{(by $(i)$)} = 0.$$

The proof of the last statement can be obtained with the arguments we gave in Remark \ref{remark extra conclusions for main theorem 1}.
\end{proof}

As we commented before, for a C$^*$-algebra $A$, the proofs of \cite[Lemma 2.10, Theorem 2.11]{AyuKudPe2014} remain valid to characterize when a bounded linear operator $T: A\to A^{**}$ is a generalized derivation. So, the proofs of the previous Theorem \ref{t maps r-*-anti-derivable at zero} remains valid to get the following result.

\begin{thm}\label{t maps r-*-anti-derivable at zero bidual} Let $T: A\to A^{**}$ be a bounded linear operator, where $A$ is a C$^*$-algebra. Then the following statements are equivalent:\begin{enumerate}[$(a)$]\item $T$ is r-$^*$-anti-derivable at zero {\rm(}i.e., $ab^* = 0 \hbox{ in } A \Rightarrow T(b)^* a + b^* T(a) = 0$ in $A^{**}${\rm)};
\item There exists a $^*$-derivation $d:A\to A^{**}$ and an element $\xi \in A^{**}$ satisfying the following properties: \begin{enumerate}[$(i)$]\item  $d([a,b]) + [a,b] \xi + \xi^* [a,b]=0,$ for all $a,b\in A$;
\item $T(a b) = a T(b) + T(a) b - a \xi b,$ and $T(a) = d(a) + a \xi $ for all $a,b\in A$.
\end{enumerate}
\end{enumerate}
\end{thm}

The description of those continuous linear operators which are l-$^*$-anti-derivable at zero is a straight consequence of the previous Theorem \ref{t maps r-*-anti-derivable at zero}.

\begin{thm}\label{t maps l-*-anti-derivable at zero} Let $T: A\to X$ be a bounded linear operator where $A$ is a C$^*$-algebra and $X$ is an essential Banach $^*$-$A$-bimodule. Then the following statements are equivalent:\begin{enumerate}[$(a)$]\item $T$ is l-$^*$-anti-derivable at zero {\rm(}i.e., $a^*b = 0 \hbox{ in } A \Rightarrow T(b) a^* + b T(a)^* = 0$ in $X${\rm)};
\item There exists a $^*$-derivation $d:A\to X^{**}$ and an element $\eta \in X^{**}$ satisfying the following properties:\begin{enumerate}[$(i)$]\item  $d([a,b]) + [a,b] \eta^* + \eta [a,b]=0,$ for all $a,b\in A$;
\item $T(a b) = a T(b) + T(a) b - a \eta b,$ and $T(a) = d(a) + \eta a $ for all $a,b\in A$.
\end{enumerate}
\end{enumerate}

If we further assume that $A$ is unital or $X$ is a dual Banach $A$-bimodule, we can replace $(b)$ with the following:
\begin{enumerate}[$(b')$]\item There exists a $^*$-derivation $d:A\to X$ and an element $\eta \in X$ satisfying the properties $(i)$-$(ii)$ above.
\end{enumerate}
\end{thm}

\begin{proof} By observing that $T$ is l-$^*$-anti-derivable at zero if and only if the mapping $S : A\to X$, $S(a) :=T(a^*)^*$ ($\forall a\in A$) is r-$^*$-anti-derivable at zero, an application of Theorem \ref{t maps r-*-anti-derivable at zero} tells that this is the case if and only if there exists a $^*$-derivation $d:A\to X^{**}$ and an element $\xi \in X^{**}$ satisfying the following properties: \begin{enumerate}[$(i)$]\item  $d([a,b]) + [a,b] \xi + \xi^* [a,b]=0,$ for all $a,b\in A$;
\item $S(a b) = a S(b) + S(a) b - a \xi b,$ and $S(a) = d(a) + a \xi $ for all $a,b\in A$.
\end{enumerate} Taking $\eta = \xi^*$ the rest can be straightforwardly checked by the reader.
\end{proof}

\begin{corollary}\label{c maps l-*-anti-derivable at zero bidual} Let $T: A\to A^{**}$ be a bounded linear operator where $A$ is a C$^*$-algebra. Then the following statements are equivalent:\begin{enumerate}[$(a)$]\item $T$ is l-$^*$-anti-derivable at zero;
\item There exists a $^*$-derivation $d:A\to A^{**}$ and an element $\eta \in A^{**}$ satisfying the following properties: \begin{enumerate}[$(i)$]\item  $d([a,b]) + [a,b] \eta^* + \eta [a,b]=0,$ for all $a,b\in A$;
\item $T(a b) = a T(b) + T(a) b - a \eta b,$ and $T(a) = d(a) + \eta a $ for all $a,b\in A$.
\end{enumerate}
\end{enumerate}
\end{corollary}

\medskip
\medskip

\textbf{Acknowledgements} A.M. Peralta partially supported by the Spanish Ministry of Science, Innovation and Universities (MICINN) and European Regional Development Fund project no. PGC2018-093332-B-I00, Junta de Andaluc\'{\i}a grant FQM375 and Proyecto de I+D+i del Programa Operativo FEDER Andalucia 2014-2020, ref. A-FQM-242-UGR18. \smallskip

This work was supported by the Deanship of Scientific Research (DSR), King Abdulaziz University. The authors, therefore, gratefully acknowledge DSR technical and financial support.\smallskip

The results in this paper are part of the first author's Ph.D. thesis at King Abdulaziz University.\smallskip

We acknowledge the thorough revision made by the anonymous referee including several sharp comments on Theorem 6.

\end{document}